\documentclass[reqno]{amsart}
\usepackage[T1]{fontenc}

\usepackage{amssymb, amsthm}

\newtheorem{theorem}{Theorem}
\newtheorem{lemma}{Lemma}
\newtheorem{proposition}{Proposition}
\newtheorem{corollary}{Corollary}
\newtheorem{remark}{Remark}

\usepackage[colorlinks,linkcolor=black,urlcolor=blue,citecolor=black]{hyperref} 

\usepackage{stmaryrd}
\usepackage{mathtools}

\newcommand{\bfloor}[1]{\bigg\lfloor #1 \bigg\rfloor}

\makeatletter
\newcommand{\subalign}[1]{%
  \vcenter{%
    \Let@ \restore@math@cr \default@tag
    \baselineskip\fontdimen10 \scriptfont\tw@
    \advance\baselineskip\fontdimen12 \scriptfont\tw@
    \lineskip\thr@@\fontdimen8 \scriptfont\thr@@
    \lineskiplimit\lineskip
    \ialign{\hfil$\m@th\scriptstyle##$&$\m@th\scriptstyle{}##$\hfil\crcr
      #1\crcr
    }%
  }%
}
\makeatother

\everymath{\displaystyle}

\title[Divisibility of the multiplicative order]{Divisibility of the multiplicative order modulo monic irreducible polynomials over finite fields}
\author{Joaquim Cera Da Concei\c c\~ao}
\address{Normandie Universit\'e, UNICAEN, CNRS, LMNO, 14000 Caen, France}
\email{\tt joaquim.cera-daconceicao@unicaen.fr}
\urladdr{\href{https://jceradaconceicao.github.io}{\tt https://jceradaconceicao.github.io}}


\begin{document}

\begin{abstract}
We consider the set of monic irreducible polynomials $P$ over a finite field $\mathbb{F}_q$ such that 
the multiplicative order modulo $P$ of some $a$ in $\mathbb{F}_q(T)$ is divisible by a fixed positive integer $d$. Call
$R_q(a,d)$ this set. We show the existence or non-existence of the density of $R_q(a,d)$ for
three distinct notions of density. In particular, the sets $R_q(a,d)$ have a Dirichlet density.
Under some assumptions, we prove simple formulas for the density values.
\end{abstract}

\maketitle
\tableofcontents

\section{Introduction}\label{SectionIntro}

Let $a\in \mathbb{Q}\setminus \{0,\pm 1\}$ and $d$ be a positive integer. The proportion of rational prime numbers $p$ such that $d$ divides the multiplicative order of $a$ modulo $p$ has been widely studied. It was originally considered by Hasse \cite{Ha1, Ha2} in $1965$ and $1966$, with $a$ a square-free integer and $d$ a prime number, in order to find the natural density of the set of primes dividing the sequence $(1 + a^n + \cdots + a^{(d-1)n})_{n\geq 0}$ at some $n\geq 0$. The problem was completely solved when Wiertelak \cite{Wie} proved the following theorem:
\begin{theorem}
    Let $N_a(d)$ be the set of prime numbers $p$ such that $d$ divides the multiplicative order of $a$ modulo $p$ and $N_a(d;x)=\#N_a(d)\cap [1,x]$. Then
    \[
    N_a(d;x) = \delta_a(d)\mathrm{Li}(x) + \mathcal{O}_{d,a}\bigg( \frac{x(\log{\log{x}})^{\omega(d)+1}}{(\log{x})^3}  \bigg),
    \]
    where the implied constant depends on $d$ and $a$, $\mathrm{Li}$ is the logarithmic integral function, $\omega$ is the number of distinct-prime-divisor function, and $\delta_a(d)\in[0,1]$ is the natural density of $N_a(d)$.
\end{theorem}
Moreover, Wiertelak gave a formula for $\delta_a(d)$ that shows that $\delta_a(d)\in \mathbb{Q}_{>0}$. More recently, Pappalardi \cite{Pap} took a different approach to this problem and obtained another equivalent formula for the density. It was given in a more compact form by Moree \cite{Moree}, using a similar method. Note that in some cases, the set $N_a(d)$ is also related to sets of prime divisors of some linear integral sequences. See \cite[Chapter 3]{Ba0} for the definition of such sequences and the computation of some of the densities.

There is an analogy between $\mathbb{Z}$ and the ring of polynomials $A=\mathbb{F}_q[T]$ with positive characteristic $p$. Both are euclidean rings in which prime numbers and monic irreducible polynomials are prime elements. The analogue of $\mathbb{Q}$ is the fraction field of $A$, that is, $K=\mathbb{F}_q(T)$. With the above problem in mind, it is natural to ask whether a similar investigation can be conducted for monic irreducible polynomials in $A$. This is the object of our paper. Let $a\in K^\times$ and $d$ be a positive integer. We define $R_q(a,d)$ to be the set of primes $P\in A$ such that the multiplicative order of $a$ in the group $(A/P)^\times$ is divisible by $d$. To our knowledge, the only instance of study of these sets are recent papers of Ballot \cite{Ba1, Ba2} in which the case $a=T$ and $d$ a prime number is treated with elementary methods. To determine the proportion of such primes, we use prime densities on $A$. Let $S\subset A$ be a set of monic irreducible polynomials and $S(N)$ be the number of $P\in S$ with polynomial degree $N$, $N\geq 1$. The most common densities are the $d_1$ and $\delta$ densities defined, when they exist, by the limits
\[
d_1(S) = \lim_{N\to +\infty} \frac{S(N)}{\mathcal{P}_+(N)} \quad \text{and} \quad \delta(S) = \lim_{s\to 1^+} \frac{\sum_{P\in S} NP^{-s}}{\sum_{P\in \mathcal{P}_+} NP^{-s}}.
\]
The letter $\mathcal{P}_+$ denotes the set of monic irreducible polynomials in $A$ and $NP=q^{\deg(P)}$ is the norm of $P$. The quantity $\mathcal{P}_+(N)$ is usually denoted $I_N$ and is given by the sum
\[
I_N= \frac{1}{N} \sum_{d\mid N} \mu(d) q^{N/d}.
\]
The number $\delta(S)$ is called the {\it Dirichlet density} of $S$ and is the analogue of the Dirichlet density used for rational prime numbers. However, in a discussion about prime densities on $A$, Ballot \cite{Ba3} defines five densities $d_1, d_2, d_3, d_4$ and $\delta$, and  concludes two things. Denoting by $\delta_1\implies \delta_2$ the fact that any set of primes in $A$ having a $\delta_1$-density equal to $d$ must have a $\delta_2$-density equal to $d$, \cite[Theorem A]{Ba3} states the following:
\[
d_1 \iff d_2 \implies d_3 \implies d_4 \iff \delta.
\]
Moreover, $d_3$ is not equivalent to $d_2$, nor $d_4$. In conclusion, there are three distinct densities to be considered. In this paper, we consider $d_1$, $\delta$ and the $d_3$-density defined by the limit
\[
d_3(S)=\lim_{N\to +\infty} \frac{1}{N} \sum_{n=1}^N \frac{S(n)}{I_n} =\lim_{N\to +\infty} \frac{1}{N} \sum_{n=1}^N \frac{S(n)}{q^n/n}, 
\]
when it exists. Note that the second equality comes from the well-known equivalence $I_n\sim q^n/n$ as $n$ tends to infinity. Secondly, although there is some evidence of $d_1$ being an analogue of the natural density commonly used on $\mathbb{N}$, Ballot concludes that $d_3$ seems to be a better candidate. Indeed, various sets of rational prime numbers that are known to have natural density have analogues in $A$ that do not have $d_1$-density but have $d_3$-density. In this work, we prove that the set $R_q(a,d)$ does not usually have $d_1$-density, but always has $d_3$-density, thus confirming $d_3$ as a strong analogue of the natural density. Our work is based on the method used by Pappalardi \cite{Pap} and Moree \cite{Moree}, and on the elementary approach taken by Ballot \cite{Ba1, Ba2}.

Furthermore, this problem is closely related to Artin's primitive root conjecture over function fields. Let $a\in K\setminus \mathbb{F}_q$ not be an $l$-th power for all $l\mid q-1$. The conjecture states that there exist infinitely many $P\in \mathcal{P}_+$ such that $a$ is a primitive root modulo $P$, that is, $a$ has order $NP-1$ in $(A/P)^\times$. The conjecture has been widely studied in the function field setting. It was first proven by Bilharz \cite{Bil} under the generalized Riemann hypothesis for function fields, which was later proved by Weil \cite{Wei}. In his proof, Bilharz shows that the set $S(a)$ of primes $P$ that have $a$ as a primitive root has positive Dirichlet density. Another proof was given in 1994 by Pappalardi and Shparlinski \cite{PapShp} by estimating the number of $P\in \mathcal{P}_+$ of degree $n\geq1$ that satisfy the conjecture. More recently, Kim and Murty \cite{KimMur} managed to prove Artin's conjecture without using the generalized Riemann hypothesis for function fields. Moreover, note that $S(T)$ has positive $d_3$-density. It was shown by Ballot \cite{Ba3} through a theorem of Shparlinski \cite[Theorem 3]{Shp}. Using Theorem \ref{MainTHMProportion} with $d=q^N-1$ for all $N\geq 1$, we see that $d_3(S(a))$ can be estimated with \cite[Theorem 3]{Shp}, or some variation of it, in a similar way, thus showing that $S(a)$ has positive $d_3$-density.

In Section \ref{SectionDef}, we give definitions and results related to function fields that make up our main toolbox. An important tool is an analogue of the Chebotarev Density Theorem for global function fields. Another gives necessary and sufficient conditions for primes in a global function field $K$ to completely split in a Galois extension $L/K$.

The idea is to describe $R_q(a,d)$ as a union of sets of primes that completely split in Kummer extensions of $K=\mathbb{F}_q(T)$, i.e., fields of the form $K(\zeta_n, a^{1/d})$, where $d,n$ are positive integers such that $(n,p)=1$ and $d\mid n$, $\zeta_n\in \Bar{\mathbb{F}}_q$ is a primitive $n$-th root of unity and $a\in K^\times$. From this and the above-mentioned analogue of Chebotarev's Density Theorem, we obtain an asymptotic formula for $R_q(a,d,N):=R_q(a,d)(N)$ that involves degrees of Kummer extensions. Therefore, to make our formula simpler and to later compute a closed-form formula for the density, we need to determine the degree of Kummer extensions of $K$. In Section \ref{SectionConstantFields}, we study the form an element $a\in K^\times$ may take when $K(a^{1/n})/K$ is a constant field extension, i.e., an extension of $\mathbb{F}_q$. These results are our primary tools for computing degrees of Kummer extensions in Section \ref{SectionKummer}.

In Section \ref{SectionProportion}, we show that $R_q(a,d,N)$ may be expressed in terms of the cardinality of sets of primes that completely split in some Kummer extensions of $K$. Applying the analogue of the Chebotarev Density Theorem, we find an asymptotic formula for $R_q(a,d,N)$ of the form
\[
|R_q(a,d,N)-\delta_q(a,d,N)\cdot q^N/N| \ll f(N),
\]
for some function $f$ and where $\delta_q(a,d,N)$ is the {\it proportion-density} of $R_q(a,d,N)$. (See Theorem \ref{MainTHMProportion}.) We obtain a formula for $\delta_q(a,d,N)$ that involves degrees of Kummer extensions of $K$.

Section \ref{SectionMorePreliminaries} is dedicated to preliminary results for the proofs of the main theorems. We prove a formula for the multiplicative order of an integer, and another for integers of the form $q^n-1$, where $n,q\geq 1$ are integers. Moreover, we give a property of the degree of constant fields of some special Kummer extensions.

Our main results, Theorem \ref{THMRqNod1d2} and Theorem \ref{MAINTHM}, on the existence or non-existence of the $d_1$ and $d_3$-densities of $R_q(a,d)$ are proved in Section \ref{SectionMainTheorems}. The asymptotic formula given in Theorem \ref{MAINTHM} revolves around a technique used by Ballot that consists of partitioning $\mathbb{N}$ into adequate disjoint arithmetic progressions. For all $n\geq 1$ that belong in the same arithmetic progression, we find that $\delta_q(a,d,n)$ is a constant independent of $n$, thus simplifying most calculations.

Under some assumption on the degree of constant fields of some Kummer extensions, we prove in Section \ref{SectionClosedD3} that $d_3(R_q(a,d))$ can be written in a closed-form formula. (See Theorem \ref{THMClosedFormd3}.)

Throughout this paper, the letters $l$ and $p$ denote prime numbers, the letter $q$ denotes a power of $p$, and the letters $d,n$ and $N$ denote positive integers with $d\mid n$ and $p\nmid n$. Given an integer $d$, we let $d^\infty$ denote the {\it supernatural number}
\[
d^\infty = \prod_{l\mid d} l^\infty.
\]
This notation allows us to consider positive divisors $v$ of $d^\infty$, i.e., $v\mid d^t$ for some $t\geq 1$, and to use notation such as
\[
(k,d^\infty) = \prod_{l\mid d} l^{v_l(k)},
\]
where $(a,b)$ denotes the gcd of $a$ and $b$, and $v_l$ denotes $l$-adic valuation. Note that $k\mapsto (k,d^\infty)$ is completely multiplicative, while $k\mapsto (k,d)$ is only multiplicative. We write $[a,b]$ for the lcm of two integers $a$ and $b$. Given a field $F$, we denote by $(F^\times)^k$ the set of $k$-th powers in $F$ and by $\Bar{F}$ its algebraic closure. We write $\omega$, $\tau$, $\mu$, $\varphi$ and $\psi$ to denote the number of distinct prime factors function, the number of divisors function, the Möbius function, Euler's totient function and Dedekind psi function, respectively. For a multiplicative group $G$ and $g\in G$, we let $\mathrm{ord}_G(g)$ and $\mathrm{ind}_G(g)$ denote respectively the order and the index of $g$. Particular cases include $G=(\mathbb{Z}/n\mathbb{Z})^\times$, for which we use the notation $\mathrm{ord}_n(g)$ and $\mathrm{ind}_n(g)$, and $G=(\mathbb{F}_q[T]/(P))^\times$, with the notation $\mathrm{ord}_P(g)$ and $\mathrm{ind}_P(g)$, where $n\geq 1$ and $P\in \mathbb{F}_q[T]$. We let the letters $K$ and $A$ denote respectively the rational function field $\mathbb{F}_q(T)$ and its integer ring $\mathbb{F}_q[T]$. For $f\in A$ non-zero, we let $\Tilde{f}$ denote the monic part of $f$, that is, the unique monic polynomial in $A$ such that $f=u \Tilde{f}$ for some $u\in\mathbb{F}_q^\times$.


\section{Known results}\label{SectionDef}

We use this section to state two important results for our work. The first theorem is a special case of a theorem that takes various forms in literature. It is usually referred to as the Chebotarev Density Theorem for global function fields. (See \cite[Proposition 6.4.8]{FriJar}.) It gives a bound on the number of primes in $K$ of a fixed degree satisfying a certain property. Since there are finitely many primes of degree $N$, the name ``density'' is not the most accurate. We use the term {\it proportion-density} instead to refer to the ``density'' number described in the theorem. We denote by $g_L$ the genus of a field $L$ and by $\mathcal{P}_+$ the set of monic irreducible polynomials in $A$.

\begin{theorem}\label{Chebo}
    Let $L/K$ be a Galois extension of global function fields and $\mathbb{F}_{q^n}$ be the constant field of $L$. Put $m:=[L:\mathbb{F}_{q^n}K]$ and
    \[
    \pi(N) :=\# \{ P\in \mathcal{P}_+ : \deg(P)=N \text{ and } P \text{ splits completely in } L \}.
    \]
    Then $\pi(N)=0$ if $n\nmid N$, and otherwise, we have
    \[
    \bigg| \pi(N)  - \frac{q^N}{Nm} \bigg| \leq \frac{2}{N m} \bigg( (m+g_L)q^{N/2} + m q^{N/4}+g_L+m \bigg).
    \]
\end{theorem}

The second result gives necessary and sufficient conditions for a prime to split completely in some Kummer extensions of $K$.

\begin{theorem}\label{splitting}
    Let $n\geq 1$ be an integer with $p\nmid n$ and $a\in K^\times$. A prime $P\in A$ such that $v_P(a)=0$ splits completely in $K(\zeta_n,a^{1/d})$ if and only if
    \[
    NP\equiv 1 \pmod{n} \quad \text{and} \quad a^{\frac{NP-1}{d}} \equiv 1 \pmod{P}.
    \]
\end{theorem}

\begin{proof}
    It suffices to follow the proof of \cite[Proposition 10.6]{Ro}.
\end{proof}


\section{On constant field extensions}\label{SectionConstantFields}

Let $M/L$ be an algebraic extension of function fields. We say that $M$ is a constant field extension of $L$ if $M=(M\cap \Bar{\mathbb{F}}_q)L$. Similarly, we say that $M$ is a geometric extension of $L$ if $M\ne L$ and $M\cap \Bar{\mathbb{F}}_q=L\cap \Bar{\mathbb{F}}_q$. Note that it is likely that $M/L$ is neither a geometric nor a constant field extension, but we can always split it in two such extensions. In this section, we give necessary and sufficient conditions for an algebraic extension $K(a^{1/n})/K$, $a\in K^\times$, to be a constant field extension.

\begin{theorem}\label{iffgeoTHM}
    Let $a\in K^\times$. Then, $K(a^{1/n})/K$ is a constant field extension if and only if $a=\mu b^n$ for some $b\in K^\times$ and $\mu\in \mathbb{F}_q^\times$.
\end{theorem}

\begin{proof}
    The ``if'' part is trivial. For the converse, we start with the case $a\in A$. First assume that the result holds for prime powers and write $n=q_1\cdots q_s$, where the $q_i$'s are powers of distinct primes and $s\geq 2$. We have
    \[
    a=\mu_1 b_1^{q_1}=\cdots =\mu_s b_s^{q_s} \quad \text{and} \quad \Tilde{a}= \Tilde{b}_1^{q_1} = \cdots = \Tilde{b}_s^{q_s},
    \]
    for some non-zero $b_i\in A$ and $\mu_i\in \mathbb{F}_q^\times$. Since the $q_i$'s are powers of distinct primes, we see that $\Tilde{b}_1 \in (K^\times)^{q_i}$ for all $1\leq i \leq s$. Thus, $\Tilde{a} = b^n$ for some non-zero $b\in A$, and $a=\mu \Tilde{a} = \mu b^n$. Hence it suffices to prove the result for prime powers. Let $l$ be a prime number and $k\geq 1$. We proceed by induction on $k\geq 1$ to show the statement holds for all $n=l^k$. The base case follows from \cite[Lemma 3.3]{HoWa}. Assume the statement holds for some $k\geq 1$ and that $K(a^{1/l^{k+1}})/K$ is a constant field extension. In particular, $K(a^{1/l})/K$ a constant field extension and we may write $a=\mu b^l$ by \cite[Lemma 3.3]{HoWa}. Let $x$ be an $l^{k+1}$-th root of $a$ in $L:=K(a^{1/l^{k+1}})$. Then
    \[
    \Tilde{a}=\Tilde{b}^l = \Tilde{x}^{l^{k+1}},
    \]
    and we obtain $\Tilde{b} = \zeta_l \Tilde{x}^{l^k}$ for some $l$-th root of unity $\zeta_l$. But $\Tilde{b}$ and $\Tilde{x}$ are both monic polynomials in the ring of integers of $L$, so that $\zeta_l=1$. Because $\Tilde{x}\in L$, we find that $K(\Tilde{b}^{1/l^k})$ is a subfield of $L$, thus $K(\Tilde{b}^{1/l^k})/K$ is a constant field extension. By the induction hypothesis, we have $\Tilde{b} = \lambda c^{l^k} = \Tilde{c}^{l^k}$ for some $\lambda\in \mathbb{F}_q^\times$ and $c\in K^\times$. Hence $\Tilde{a} = \Tilde{b}^l = \Tilde{c}^{l^{k+1}}$ and $a=\mu \Tilde{c}^{l^{k+1}}$, where $\mu \in \mathbb{F}_q^\times$ is the leading coefficient of $a$. We successfully proved the result for $a\in A$. If $a=f/g \in K^\times$, then $K(a^{1/n}) = K((fg^{n-1})^{1/n})$ and by the above, $fg^{n-1}$ is of the form $\mu b^n$. Hence $a=\mu (b/g)^n$.
\end{proof}


\section{The degree of Kummer extensions of rational function fields}\label{SectionKummer}

We extend the notation $\Tilde{a}$ to rational functions $a=f/g\in K^\times$ by $\Tilde{a} =\Tilde{f}/\Tilde{g} $. We denote by $\lambda\in \mathbb{F}_q^\times$ the unique constant such that $a=\lambda \Tilde{a}$. Let $h$ denote the largest integer $t\geq 1$ such that $\Tilde{a} \in (K^\times)^t$. In this section, we study Kummer extensions of $K=\mathbb{F}_q(T)$, i.e., fields of the form $K(\zeta_n,a^{1/d})$, where $\zeta_n \in \Bar{\mathbb{F}}_q$ is a primitive $n$-th root of unity and $a \in K^\times$. We find formulas for the following field degrees:
\[
[K(\zeta_n,a^{1/d}):\mathbb{F}_{n,d}K] \quad \text{and} \quad  [\mathbb{F}_{n,d}:\mathbb{F}_q],
\]
where $\mathbb{F}_{n,d}$ denotes the constant field of $K(\zeta_n,a^{1/d})$. The key result we use is the following theorem on polynomials of the form $X^n-a$:

\begin{theorem}\label{ThmLangIrr}
    Let $K$ be a field and $a\in K^\times$. Then $X^n-a$ is irreducible over $K$ if and only if $a\not \in (K^\times)^l$ for all $l\mid n$ and $a\not \in -4 (K^\times)^4$ if $4\mid n$.
\end{theorem}

\begin{proof}
    See \cite[Theorem 9.1]{Lang}.
\end{proof}

\begin{lemma}\label{degreKummerNonGeo}
    Let $\mu\in \mathbb{F}_q^\times$. We have
    \[
    [\mathbb{F}_q(\zeta_n,\mu^{1/d}):\mathbb{F}_q] = \frac{\mathrm{ord}_n(q) d}{(\mathrm{ind}_{\mathbb{F}_q(\zeta_n)^\times}(\mu),d)}.
    \]
\end{lemma}

\begin{proof}
    It is known that $ [\mathbb{F}_q(\zeta_n):\mathbb{F}_q]=\mathrm{ord}_n(q)$. Let $u$ be the index of $\mu$ in $\mathbb{F}_q(\zeta_n)^\times$, that is, the greatest positive divisor $t \mid q^{\mathrm{ord}_n(q)}-1$ such that $\mu=x^t$ for some $x\in \mathbb{F}_q(\zeta_n)^\times$. We have $\mathbb{F}_q(\zeta_n,\mu^{1/d}) = \mathbb{F}_q(\zeta_n,v^{1/d_0})$, where $d_0=d/(d,u)$ and $v^{(u,d)}=\mu$. We claim that $d_0$ is the degree of the extension $\mathbb{F}_q(\zeta_n,\mu^{1/d})/\mathbb{F}_q(\zeta_n)$. Indeed, let us show that $X^{d_0}-v$ is irreducible over $\mathbb{F}_q(\zeta_n)$ using Theorem \ref{ThmLangIrr}. Let $l\mid d_0$ be a prime. By contradiction, if we have $v=c^l$ for some $c\in \mathbb{F}_q(\zeta_n)$, then
    \[
    x^u = \mu = v^{(u,d)} = c^{l(u,d)}.
    \]
    Because $d_0\mid q^{\mathrm{ord}_n(q)}-1$ and by the maximality of $u$, we have $l(u,d) \mid u$. This yields a contradiction since $l\mid d_0$. Now, if $4\mid d_0$, assume by contradiction that $v=-4y^4$ for some $y\in \mathbb{F}_q(\zeta_n)$. Then, since $v$ is not a square in $\mathbb{F}_q(\zeta_n)$ by the above, we find that $-1$ is not a square in $\mathbb{F}_q(\zeta_n)$. Hence $4\nmid q^{\mathrm{ord}_n(q)}-1$, but $4\mid d_0$ and $d_0\mid n$ imply that $4\mid q^{\mathrm{ord}_n(q)}-1$. A contradiction.
\end{proof}

\begin{lemma}\label{LemmaConstantsFieldKnd}
    Let $a\in K^\times$. Then the greatest divisor $v$ of $d$ such that the extension $K(\zeta_n,a^{1/v})/K$ is a constant field extension is equal to $(d,h)$. In particular, we have $\mathbb{F}_{n,d}=\mathbb{F}_q(\zeta_n,\lambda^{1/(d,h)})$.
\end{lemma}

\begin{proof}
    We claim that $K(\zeta_n,a^{1/D})$, with $D=(d,h)$, is the maximal subfield $M$ of $K(\zeta_n,a^{1/d})$ such that $M/K$ is a constant field extension. Write $h=Dk$ and $\Tilde{a}=\Tilde{b}^{Dk}$ for some $k\geq 1$ and $\Tilde{b} \in K^\times$. Then
    \[
    K(\zeta_n,a^{1/D})=K(\zeta_n,\lambda^{1/D} \Tilde{b}^k) = K(\zeta_n,\lambda^{1/D}) = \mathbb{F}_{q}(\zeta_n,\lambda^{1/D}) K
    \]
    is a constant field extension of $K$. Next, we prove that $D$ is maximal. By contradiction, assume there exists a prime $l$ such that $lD\mid d$ and that $ K(\zeta_n,a^{1/lD})$ is a constant field extension of $K$. Then, $K(a^{1/lD})/K$ is a constant field extension and
    \[
    a=\omega c^{lD} = \lambda \Tilde{c}^{lD},
    \]
    for some $\omega, \lambda\in \mathbb{F}_q^\times$ and $c\in K^\times$, by Theorem \ref{iffgeoTHM}. Hence $\Tilde{a}=\Tilde{c}^{lD}$ and, by maximality of $h$, we find that $lD \mid h$. A contradiction to $lD\mid d$.
\end{proof}

\begin{theorem}\label{degreeKummerTHM}
    Let $a\in K^\times$. We have $[K(\zeta_n,a^{1/d}):\mathbb{F}_{n,d}K] = d/(d,h)$.
\end{theorem}

\begin{proof}
    Put $d_0=d/(d,h)$ and write $a=b^D$, where $D=(d,h)$ and $b\in \mathbb{F}_{n,d}K^\times$. The latter is possible since $a=\lambda \Tilde{c}^D$ for some $c \in K^\times$, by Lemma \ref{LemmaConstantsFieldKnd} and Theorem \ref{iffgeoTHM}. Thus $a=(\mu \Tilde{c})^D$ with $\mu \in \mathbb{F}_{n,d}$ and $\mu^D=\lambda$. Using Theorem \ref{ThmLangIrr}, we show that the polynomial $X^{d_0}-b$ is irreducible over $\mathbb{F}_{n,d}K^\times$. Let $l\mid d_0$ be a prime and assume by contradiction that $b=x^l$ for some $x\in \mathbb{F}_{n,d}K^\times$. Since $\mathbb{F}_{n,d}K$ is a rational function field and because $a\in A$, we find that $x\in \mathbb{F}_{n,d}[T]$. Then $\lambda= x_0^{lD}$, where $x_0\in \mathbb{F}_{n,d}$ is the leading coefficient of $x$. However, by Lemma \ref{LemmaConstantsFieldKnd}, $D$ is the greatest positive integer $t\mid d$ such that $\lambda$ is a $t$-th power in $\mathbb{F}_{n,d}^\times$. Hence we have a contradiction and $b\not \in (\mathbb{F}_{n,d}K^\times)^l$. If $4\mid d_0$, then the rest of the proof follows in the same way as the proof of Lemma \ref{degreKummerNonGeo}.
\end{proof}


\section{The proportion-density}\label{SectionProportion}

Throughout this paper, we denote by $\mathcal{P}_+$ the set of monic irreducible polynomials in $A$, and by $R_q(a,d)$ the set of $P\in \mathcal{P}_+$ that satisfy $v_P(a)=0$ and $d\mid \mathrm{ord}_P(a)$, where $a\in K^\times$ and $d\geq 1$ are fixed. For each $N\geq 1$, we consider the number of primes in $R_q(a,d)$ with degree $N$, denoted $R_q(a,d,N)$. We write $I_N=\mathcal{P}_+(N)$.

For $d=1$, we trivially have $R_q(a,1,N)=I_N-a_N$, where $a_N$ is the number of $P$ of degree $N$ such that $v_P(a)\ne 0$. Note that $a_N\ne 0$ for only finitely many $N$. Moreover, if $a=\lambda \in \mathbb{F}_q^\times$, then
\[
R_q(\lambda,d,N) =
    \begin{cases}
        I_N, & \text{if } d\mid \mathrm{ord}_{\mathbb{F}_q^\times}(\lambda); \\
        0, & \text{otherwise}.
    \end{cases}
\]
Densities are easily computed here. We have $d_1(R_q(a,1))=1$ and it follows that the $d_i$'s and the Dirichlet densities also exist and equal $1$. The same goes for $R_q(a,\lambda)$ for which the densities are $1$ or $0$, whether $d\mid \mathrm{ord}_{\mathbb{F}_q^\times}(\lambda)$.

We assume that $d\geq 2$ and $a$ is not a constant for the rest of this paper. Since $f:=\mathrm{ord}_d(q)\nmid N$ implies that $R_q(a,d,N)=0$, we only consider integers $N\equiv 0\pmod{f}$. We put
\[
e_N(d)=e_N:=\bigg( \frac{q^N -1}{d}, d^\infty \bigg).
\]
Since $K$ is fixed, we denote by $\{L\}$ the set of primes in $K$ that splits completely in $L$, where $L/K$ is an algebraic extension. For each $N\geq 1$, we denote the number of primes in $\{L\}$ with degree $N$ by $\{L\}_N$. We assume throughout the paper that $a=\lambda\Tilde{a}$ and $\Tilde{a} = b^h$, with the notation of the previous section.

\begin{lemma}\label{keyIddegN}
    For each positive $N\equiv 0 \pmod{f}$, we have
    \[
    R_q(a,d,N) = \sum_{v\mid e_N } \sum_{u\mid d} \mu(u)\{ K_{dv,uv} \}_N,
    \]
\end{lemma}

\begin{proof}
    We follow the proof of \cite[Proposition 1]{Moree}. Note in particular that the condition $p\leq x$ must be replaced by $\deg(P)=N$, and $p\equiv 1 \pmod{dv}$ by the condition $q^N\equiv 1 \pmod{dv}$.
\end{proof}

\begin{proposition}\label{boundgenus}
    There exists $c_0>0$, that only depends on $a$, such that
    \[
    g_{K_{n,d}} \leq c_0 \cdot [K_{n,d}:\mathbb{F}_{n,d} K].
    \]
\end{proposition}

\begin{proof}
    It suffices to apply \cite[Proposition 3.7.3]{Sti} to $K_{n,d}$.
\end{proof}

\begin{lemma}\label{LemmaAppCheboKnd}
    There exists an absolute constant $c_1>0$ such that for each $N\geq 1$ such that $[\mathbb{F}_{n,d}:\mathbb{F}_q]\mid N$, we have
    \[
    \bigg|\{ K_{n,d} \}_N - \frac{q^N}{N}\frac{1}{[K_{n,d}:\mathbb{F}_{n,d}K]} \bigg| \leq 2c_1 \cdot \frac{q^{N/2}}{N}.
    \]
\end{lemma}

\begin{proof}
    We apply Theorem \ref{Chebo} to $\{K_{n,d}\}_N$. The error term is obtained using Proposition \ref{boundgenus} and the bounds
    \[
    \frac{1}{N} \leq \frac{1}{\sqrt{q}} \frac{q^{N/2}}{N} \quad \text{and} \quad \frac{q^{N/4}}{N} \leq \frac{1}{\sqrt[4]{q}} \frac{q^{N/2}}{N},
    \]
    valid for all $N\geq 1$.
\end{proof}

Let $\mathcal{P}$ be a proposition. Throughout the rest of the paper, we use the Iverson symbol defined by $[\mathcal{P}]=1$ if $\mathcal{P}$ is true, and $[\mathcal{P}]=0$ otherwise. For integers $v\mid d^\infty$ and $u\mid d$, we let $f_{u,v}=[\mathbb{F}_{dv,uv}:\mathbb{F}_q]$, that is,
\begin{equation}\label{formula_f_uv}
    f_{u,v} = \frac{\mathrm{ord}_{dv}(q)(uv,h)}{(\mathrm{ind}_{\mathbb{F}_q(\zeta_{dv})^\times}(\lambda), uv,h)},
\end{equation}
by Lemmas \ref{degreKummerNonGeo} and \ref{LemmaConstantsFieldKnd}.

\begin{theorem}\label{MainTHMProportion}
    For each positive $N\equiv 0\pmod{f}$, we have
    \[
    \bigg| R_q(a,d,N) - \frac{q^N}{N} \cdot \delta_q(a,d,N) \bigg| \leq 2^{\omega(d)+1}c_1 \cdot \frac{\tau(e_N)q^{N/2}}{N},
    \]
    where $c_1$ is the constant defined in Lemma \ref{LemmaAppCheboKnd} and $\delta_q(a,d,N)$ is the proportion-density defined by
    \[
    \delta_q(a,d,N) = \sum_{v\mid e_N} \sum_{u\mid d} \frac{\mu(u)[f_{u,v}\mid N]}{[K_{dv,uv}:\mathbb{F}_{dv,uv}K]}.
    \]
\end{theorem}

\begin{proof}
    We let $S_q(a,d,N):=R_q(a,d,N) - \delta_q(a,d,N)\cdot q^N/N $. By Lemmas \ref{keyIddegN} and \ref{LemmaAppCheboKnd}, we have
    \[
    |S_q(a,d,N)| \leq 2c_1 \cdot \frac{q^{N/2}}{N} \sum_{v\mid e_N} \sum_{u\mid d} |\mu(u)| = 2^{\omega(d)+1}c_1 \cdot \frac{\tau(e_N)q^{N/2}}{N},
    \]
    the sought result.
\end{proof}


\section{More preliminary results}\label{SectionMorePreliminaries}

We show in this section three preliminary results to the study of the $d_1$ and $d_3$-densities of $R_q(a,d)$. Our first result gives basic arithmetic properties of certain numbers of the form $(q^N-1,d^\infty)$. As a consequence, we prove a formula for $\mathrm{ord}_{dv}(q)$ that generalizes a well-known formula for $\mathrm{ord}_{l^k}(q)$, with $k\geq 1$. Finally, we prove that, under the hypothesis $f_{1,v} \mid N$, the Iverson symbol $[f_{u,v}\mid N]$ defines a multiplicative function in the variable $u$. We use the letter $\mathcal{P}$ to denote the proposition
\[
\mathcal{P} : \quad 2\|d, \quad q\equiv 3\pmod{4} \quad \text{and} \quad 2\nmid f.
\]
The following lemma is enough to see that interesting things might happen when $\mathcal{P}$ is true:

\begin{lemma}\label{lemma(dinf)}
    Let $m,n,q\geq 1$ be integers with $(d,q)=1$ and $d\mid q^m-1$. Then
    \[
    (q^{mn}-1,d^\infty) = (q^m-1,d^\infty) (n,d^\infty) \cdot 
    \begin{cases}
        2^{v_2(q^m+1)-1}, & \text{if } [\mathcal{P}]=1 \text{ and } 2\mid n; \\
        1, & \text{otherwise}.
    \end{cases}
    \]
\end{lemma}

\begin{proof}
    The map $d\mapsto (k,d^\infty)$, where $k$ is a fixed integer, defines a completely multiplicative function. Thus, it suffices to prove the result for $(q^{mn}-1,l^\infty)$, where $l\mid d$. By \cite[Lemma 4]{Ba2} and \cite[Proposition 2.4]{Ba1}, and by replacing $\mathrm{ord}_l(q)$ by $m$ and $q$ by $q^m$ respectively in the proofs, which is allowed since it only uses that $l\mid q^m-1$, we obtain
    \[
    v_l\left( \frac{q^{mn}-1}{q^m-1} \right) = v_l(n) +
    \begin{cases}
        2^{v_2(q^m+1)-1}, & \text{if } l=2 \text{ and } 2\mid n; \\
        1, & \text{otherwise}.
    \end{cases}
    \]
    The result follows by noting that if $\mathcal{P}$ is false, then $v_2(q^m+1)=1$.
\end{proof}

\begin{lemma}\label{LemmaFormulaOrder}
    Let $q\geq 1$ be an integer prime to $d$ and $f=\mathrm{ord}_d(q)$. Then, for all $v\mid d^\infty$, we have
    \[
    \mathrm{ord}_{dv}(q) = fdv \cdot 
    \begin{dcases}
        \frac{2}{(q^{2f}-1,dv)}, & \text{if $[\mathcal{P}]=1$ and } 2\mid v; \\
        \frac{1}{(q^f-1,dv)}, & \text{otherwise}.
    \end{dcases}
    \]
\end{lemma}

\begin{proof}
    Assume $v_2(d)\ne 1$ and put $n=dv/(q^f-1,dv)$. By Lemma \ref{lemma(dinf)}, we see that $dv \mid (q^{fn}-1,d^\infty)$. Hence $n=tm$, where $m\geq 1$ and $t:=\mathrm{ord}_{dv}(q)/f$. By Lemma \ref{lemma(dinf)}, we have
    \[
    dv\mid (q^{ft}-1,d^\infty) = (q^f-1,d^\infty)\cdot \frac{dv}{(q^f-1,dv)(m,d^\infty)},
    \]
    and we find that $m=(m,d^\infty)$ divides $(q^f-1,d^\infty)/(q^f-1,dv)$. But the latter is coprime to $n$, which yields that $m=1$. Next, assume that $2\, \| \,d$ and note that for any odd integer $n\geq 1$, we have $\mathrm{ord}_{2n}(q) = \mathrm{ord}_n(q)$. Therefore, when $2\nmid v$, we have $\mathrm{ord}_{dv}(q) = \mathrm{ord}_{dv/2}(q)$ and we conclude using what we proved in the above. When $2\mid v$, put $D=2d$ and $u=v/2$ so that
    \[
    \mathrm{ord}_{dv}(q) = \mathrm{ord}_{Du}(q) = \frac{\mathrm{ord}_{D}(q) dv}{(q^{\mathrm{ord}_{D}(q)}-1,dv)},
    \]
    by the above again. We have $\mathrm{ord}_{D}(q)=[\mathrm{ord}_{4}(q),\mathrm{ord}_{d/2}(q)]=[\mathrm{ord}_{4}(q),f]$ because $2^2\|D$, and we conclude using that $\mathrm{ord}_4(q)$ equals $1$ or $2$, whether $q\equiv 1\pmod{4}$ or $q\equiv 3 \pmod{4}$ respectively.
\end{proof}

\begin{remark}\label{remark_fk}
    We know that $f_{u,v} = \mathrm{ord}_{dv}(q) k_0$ for some $k_0\mid d^\infty$ by \eqref{formula_f_uv}. Hence Lemma \ref{LemmaFormulaOrder} shows that $f_{u,v} = fk$ for some $k\mid d^\infty$.
\end{remark}

\begin{lemma}\label{Lemmaf_uMulti}
    Let $N\geq 1$ be such that $f_{1,v} \mid N$. The function $u\mapsto [f_{u,v}\mid N]$ is multiplicative for all $v\mid d^\infty$.
\end{lemma}

\begin{proof}
    We want to show that $[f_{u_1u_2,v}\mid N]=[f_{u_1,v}\mid N]\cdot [f_{u_2,v}\mid N]$ for coprime $u_1,u_2\mid d$, or equivalently, $f_{u_1u_2,v}=[f_{u_1,v},f_{u_2,v}]$. Let $r=\mathrm{ord}_{dv}(q)$ and $m=\mathrm{ord}_{\mathbb{F}_q^\times}(\lambda)$. We see that $\mathrm{ind}_{\mathbb{F}_q(\zeta_{dv})^\times}(\lambda) = (q^r-1)/m$ because $\lambda\in \mathbb{F}_q^\times$. Moreover, we note that
    \[
    (uv,h) =(v,h) \left( \frac{uv}{(v,h)}, \frac{h}{(v,h)} \right) = (v,h) \left(u,\frac{h}{(v,h)}\right),
    \]
    for all $u\mid d$. From the above and Lemmas \ref{degreKummerNonGeo} and \ref{LemmaConstantsFieldKnd}, we obtain
    \begin{equation}\label{eqf_uv}
        f_{u,v} = \frac{r(uv,h)}{\left( \frac{q^r-1}{m},uv,h\right)} = \frac{r(v,h)}{(v,H_r)} \cdot \frac{\left( u, \frac{h}{(v,h)} \right)}{\left( u,\frac{H_r}{(v,H_r)} \right)},
    \end{equation}
    for all $u\mid d$, where $H_r= (h,(q^r-1)/m)$. Now, we find
    \[
    [f_{u_1,v},f_{u_2,v}] = \frac{r(v,h)}{(v,H_r)} \cdot \left[ \frac{\left( u_1, \frac{h}{(v,h)} \right)}{\left( u_1,\frac{H_r}{(v,H_r)} \right)}, \frac{\left( u_2, \frac{h}{(v,h)} \right)}{\left( u_2,\frac{H_r}{(v,H_r)} \right)} \right].
    \]
    Since $(u_1,u_2)=1$, the lcm is the product of the two numbers. Moreover, recall that $u\mapsto (u,k)$ is a multiplicative function for all fixed integers $k$. Therefore, using \eqref{eqf_uv}, we find that $ [f_{u_1,v},f_{u_2,v}]=f_{u_1u_2,v}$.
\end{proof}


\section{The main theorems}\label{SectionMainTheorems}

In this section, we prove that the set $R_q(a,d)$ does not have $d_1$-density, and thus no $d_2$-density by equivalence, if $f\geq 2$. However, we show that the $d_3$-density exists and we find a formula for it. From now on, we denote by $\Bar{f}$ the order of $q$ modulo $d(h,d^\infty)$.

\begin{theorem}\label{THMRqNod1d2}
    If $f\geq 2$, then the set $R_q(a,d)$ has no $d_1$-density, nor $d_2$-density.
\end{theorem}

\begin{proof}
    It suffices to show that the limit, as $N$ tends to infinity, of
    \begin{equation}\label{limit}
        \frac{R_q(a,d,N)N}{q^N}
    \end{equation}
    does not exist. We proceed to show that \eqref{limit} converges to different limits for distinct subsequences. Let $(x_n)_{n\geq 1}$ be an increasing sequence of integers $N$ not divisible by $f$. By the discussion at the beginning of Section \ref{SectionProportion}, we have $R_q(a,d,x_n)=0$ for all $n\geq 0$. Let $(y_n)_{n\geq 0}$ be the sequence defined by $y_n = \Bar{f}(h,d^\infty)(dn+1)$ for all $n\geq 0$. By Theorem \ref{MainTHMProportion}, we have
    \begin{equation}\label{BoundEq}
        \bigg| \frac{R_q(a,d,y_n)y_n}{q^{y_n}} - \delta_q(a,d,y_n)\bigg| \leq 2^{\omega(d)+1}c_1 \tau(e_{y_n}) q^{-y_n/2},
    \end{equation}
    and by Lemma $\ref{lemma(dinf)}$, we see that $e_{y_n} = e_{\Bar{f}}(h,d^\infty)\nu$, with $\nu\mid 2^\infty$, making $e_{y_n}$ a constant, say $E$, that does not depend on $n$. In particular, $\tau(e_{y_n})$ is also a constant and the right hand side of \eqref{BoundEq} converges to $0$ as $n\to +\infty$. Hence the limit of \eqref{limit}, where we substituted $N$ for $y_n$, is the limit of $\delta_q(a,d,y_n)$. By Theorems \ref{degreeKummerTHM} and \ref{MainTHMProportion},
    \[
    \delta_q(a,d,y_n) =  \sum_{v\mid E} \sum_{u\mid d} \frac{\mu(u)(uv,h)}{uv} \cdot [f_{u,v} \mid \Bar{f}(h,d^\infty)],
    \]
    where we used that $e_{y_n}=E$ and that $f_{u,v}\mid y_n$ if and only if $f_{u,v} \mid \Bar{f}(h,d^\infty)$. The latter follows from Remark \ref{remark_fk}. It follows that $\delta_q(a,d,y_n)$ is a constant that does not depend on $n$, say $\delta$. Since $\mathbb{F}_{dv,v}$ is a subfield of $\mathbb{F}_{dv,uv}$, we find that $f_v:=f_{1,v}$ divides $f_{u,v}$ for all $u\mid d$. We may write
    \[
    \delta= \sum_{\substack{v\mid E \\ f_v\mid \Bar{f}(h,d^\infty)}} \sum_{u\mid d} \frac{\mu(u)(uv,h)}{uv} \cdot [f_{u,v} \mid \Bar{f}(h,d^\infty)].
    \]
    For all $v\mid E$ such that $f_v\mid \Bar{f}(h,d^\infty)$, the function
    \[
    u \mapsto \frac{\mu(u)(uv,h)}{u(v,h)} \cdot [f_{u,v} \mid \Bar{f}(h,d^\infty)]
    \]
    is multiplicative by Lemma \ref{Lemmaf_uMulti}. Therefore, we have
    \[
    \delta = \sum_{\substack{v\mid E \\ f_v\mid \Bar{f}(h,d^\infty)}} \frac{(v,h)}{v}  \prod_{l\mid d} \left( 1 - \frac{(lv,h)}{l(v,h)}\cdot [f_{l,v}\mid \Bar{f}(h,d^\infty)] \right),
    \]
    where we used the Euler product formula. We see that $\delta$ is non-negative as each general term is. Hence it suffices to show that there is one non-zero term. We claim that the term in $v=(h,d^\infty)$ is positive. First, let us show that we have $f_v \mid \Bar{f}(h,d^\infty)$. Note that $(v,h) = v$, so that
    \[
    f_{(h,d^\infty)} = \frac{\mathrm{ord}_{dv}(q) (v,h)}{(\mathrm{ind}_{\mathbb{F}_q(\zeta_{dv})^\times}(\lambda),v,h)} = \frac{\Bar{f}(h,d^\infty)}{(\mathrm{ind}_{\mathbb{F}_q(\zeta_{dv})^\times}(\lambda),h,d^\infty)}.
    \]
    We find that $f_{(h,d^\infty)} \mid \Bar{f}(h,d^\infty)$ and the term in $v=(h,d^\infty)$ appears in the sum. Using again that $(v,h)=v$ and that $(lv,h)=v$, we find the general term to be
    \[
     \frac{(h,d^\infty)}{v}  \prod_{l\mid d} \left( 1 - \frac{[f_{l,v}\mid \Bar{f}(h,d^\infty)] }{l} \right),
    \]
    which is non-zero, whether $f_{l,v}\mid \Bar{f}(h,d^\infty)$ or not. We obtain that \eqref{limit} converges to $0$ and $\delta>0$ for distinct subsequences. Hence, the set $R_q(a,d)$ has no $d_1$-density, and thus no $d_2$-density by \cite[Proposition 1.8]{Ba3}.
\end{proof}

The proof of the existence and the computation of the $d_3$-density of $R_q(a,d)$ requires to partition $\mathbb{N}$ into a countable union of distinct arithmetic progressions, following a method of Ballot \cite{Ba2}. We have
\begin{equation}\label{EqPartition}
    \mathbb{N} = \bigsqcup_{j=1}^{f-1} S_j \sqcup \bigsqcup_{w\mid d^\infty} \bigsqcup_{\substack{\alpha=1 \\ (\alpha,d)=1}}^d A_{w,\alpha},
\end{equation}
where $ S_j=\{ fn+j : n\geq 0 \}$ and $A_{w,\alpha} = \{ fw(\alpha+dn) : n\geq 0  \}$. For $N\in S_j$, we have $\delta_q(a,d,N)=0$ by the discussion at the beginning of Section \ref{SectionProportion}. For $N\in A_{w,\alpha}$, we have $e_N =e_{fw}$ by Lemma \ref{lemma(dinf)}, which only depends on $w$. Moreover, we have $f_{u,v} \mid N$ if and only if $f_{u,v} \mid fw$, by Remark \ref{remark_fk}. We obtain
\begin{equation}\label{EqFormulaDelta_w}
    \delta_q(a,d,N) = \sum_{v\mid e_{fw}} \sum_{u\mid d} \frac{\mu(u)(uv,h)}{uv} \cdot [f_{u,v} \mid fw]
\end{equation}
which is a constant that does not depend on $n$, nor $\alpha$. We denote this quantity by $\delta_w$. Moreover, we denote by $\delta_q(a,d)$ the sum
\begin{equation}\label{delta_q(a,d) first expression}
    \delta_q(a,d) = \frac{\varphi(d)}{df} \sum_{w\mid d^\infty} \frac{\delta_w}{w}.
\end{equation}
We show that $\delta_q(a,d)$ is the $d_3$-density of the set $R_q(a,d)$.

\begin{lemma}\label{BoundsSums}
    There exists an absolute constant $c_2>0$ such that for every $x\geq e^{2\omega(d)}$, we have
    \[
    \sum_{ \subalign{ w&\mid d^\infty \\w&\leq x } } 1 \leq c_2\log(x)^{\omega(d)} \quad \text{and} \quad \sum_{ \subalign{ w&\mid d^\infty \\w&>x } } \frac{1}{w} \leq \frac{2c_2\log{(x)}^{\omega(d)}}{x}.
    \]
\end{lemma}

\begin{proof}
    Let $M_d(x)$ denote the sum on the left. We see that $M_d(x)$ is bounded above by the product of $\log_l(x)+1$ for all $l\mid d$. We have $\log_l(x)+1\leq \log(x)$ for all primes $e^2\leq l \leq x$. If $l\leq e^2$, then there exists a constant $C_l>0$ such that $\log_l(x)+1\leq C_l\log(x)$. The result follows by choosing $c_2$ as the product of the $C_l$'s for all primes $l\leq e^2$. Next, we apply Abel summation formula to the series on the right so that $M_d(x)$ appears in the expression. We find
     \begin{align*}
        \sum_{ \subalign{ w&\mid d^\infty \\w&>x } } \frac{1}{w} &= -\frac{M_d(x)}{x} + \int_x^{+\infty} \frac{M_d(t)}{t^2} \mathrm{d}t \leq c_2 \int_x^{+\infty} \frac{\log{(t)}^{\omega(d)}}{t^2} \mathrm{d}t.
    \end{align*}
    Call $I(x)$ the integral on the right-hand side of the inequality. For $x\geq e^{2\omega(d)}$, we see that $I(x)-2\log(x)^{\omega(d)}/x$ is an increasing function that converges to $0$ as $x$ tends to infinity. Hence $I(x) \leq 2\log(x)^{\omega(d)}/x$ and the result follows.
\end{proof}

\begin{lemma}\label{LemmaSN}
    For every $N\geq e^{2\omega(d)}$, we have
    \[
    \bigg| \frac{1}{N} \sum_{n=1}^N \delta_q(a,d,n) - \delta_q(a,d) \bigg| \leq c_2 \varphi(d)\bigg(1+\frac{2}{fd}\bigg) \frac{\log{(N)}^{\omega(d)}}{N},
    \]
    where $c_2$ is the absolute constant defined in Lemma \ref{BoundsSums}.
\end{lemma}

\begin{proof}
    From the partition of $\mathbb{N}$ given in \eqref{EqPartition}, we have
    \[
    S_N := \frac{1}{N} \sum_{n=1}^N \delta_q(a,d,n) = \frac{1}{N} \sum_{w\mid d^\infty} \sum_{\substack{\alpha=1 \\ (\alpha,d)=1}}^d \sum_{n\in A_{w,\alpha}(N)} \delta_w,
    \]
    where $A_{w,\alpha}(N) = A_{w,\alpha} \cap \llbracket 1,N \rrbracket $, and using that $\delta_q(a,d,n) = 0$ if $n\in S_j$, and that $\delta_q(a,d,n)=\delta_w$ if $n\in A_{w,\alpha}$. Note that $w\leq N$ and
    \[
    \# A_{w,\alpha}(N) = \bfloor{\frac{N+fw(d-\alpha)}{fdw}},
    \]
    thus, we obtain
    \[
    S_N = \frac{1}{N} \sum_{\substack{w\mid d^\infty \\ w\leq N}} \sum_{\substack{\alpha=1 \\ (\alpha,d)=1}}^d \bfloor{\frac{N+fw(d-\alpha)}{fdw}} \delta_w.
    \]
    By definition of the floor function, on the one hand, we have
    \[
    S_N\geq \frac{\varphi(d)}{N} \sum_{\substack{w\mid d^\infty \\ w\leq N}} \delta_w\left( \frac{N}{fdw} - 1 \right)= \delta_q(a,d) -  \frac{\varphi(d)}{N} \sum_{\substack{w\mid d^\infty \\ w\leq N}} \delta_w  - \frac{\varphi(d)}{fd}\sum_{\substack{w\mid d^\infty \\ w>N}} \frac{\delta_w}{w}.
    \]
    and on the other hand,
    \[
    S_N \leq \delta_q(a,d) + \frac{\varphi(d)}{N} \sum_{\substack{w\mid d^\infty \\ w\leq N}} \delta_w  + \frac{\varphi(d)}{fd}\sum_{\substack{w\mid d^\infty \\ w>N}} \frac{\delta_w}{w}.
    \]
    We obtain
    \[
    |S_N-\delta_q(a,d) | \leq \frac{\varphi(d)}{N} \sum_{\substack{w\mid d^\infty \\ w\leq N}} \delta_w  + \frac{\varphi(d)}{fd}\sum_{\substack{w\mid d^\infty \\ w>N}} \frac{\delta_w}{w} \leq c_2\varphi(d)\left(1+ \frac{2}{fd} \right)\frac{\log(N)^{\omega(d)}}{N},
    \]
    where we used that $\delta_w\leq 1$, and Lemma \ref{BoundsSums}.
\end{proof}

\begin{theorem}\label{MAINTHM}
    There exists an absolute constant $c_3>0$ such that
    \[
    \bigg| \frac{1}{N} \sum_{n=1}^N \frac{R_q(a,d,n)}{q^n/n} - \delta_q(a,d) \bigg| \leq c_2\varphi(d)\bigg(1+\frac{2}{fd}\bigg) \frac{\log{(N)}^{\omega(d)}}{N} + \frac{c_3}{N},
    \]
    for all $N\geq e^{2\omega(d)}$, where $c_2$ is the absolute constant defined in Lemma \ref{BoundsSums}. In particular, $R_q(a,d)$ has $d_3$-density equal to $\delta_q(a,d)$.
\end{theorem}

\begin{proof}
    First, we put
    \[
    R_N = \frac{1}{N} \sum_{n=1}^N \frac{R_q(a,d,n)}{q^n/n} \quad \text{and} \quad S_N = \frac{1}{N} \sum_{n=1}^N \delta_q(a,d,n).
    \]
    We have
    \[
    |R_N - \delta_q(a,d) | \leq | R_N - S_N | + c_2\varphi(d)\bigg(1+\frac{2}{fd}\bigg) \frac{\log{(N)}^{\omega(d)}}{N},
    \]
    for all $N\geq e^{2\omega(d)}$, by Lemma \ref{LemmaSN}. Let us now bound the term $| R_N - S_N | $. Using Theorem \ref{MainTHMProportion}, and since $R_q(a,d,n)=\delta_q(a,d,n)=0$ if $f\nmid n$, we have
    \[
    | R_N - S_N | \leq \frac{1}{N} \sum_{\substack{n=1 \\ f\mid n}}^N \bigg|\frac{R_q(a,d,n)}{q^n/n} - \delta_q(a,d,n)  \bigg| \leq \frac{2^{\omega(d)+1}c_1}{N} \sum_{\substack{n=1 \\ f\mid n}}^N \tau(e_n)q^{-n/2}.
    \]
    By Lemma \ref{lemma(dinf)}, $e_n \leq 2^{v_2(q^f+1)} e_f n/f$. Hence $\tau(e_n) \leq v_2(q^f+1) \tau(e_f) n/f $, and
    \[
    | R_N - S_N | \leq \frac{2^{\omega(d)+1} v_2(q^f+1) c_1 \tau(e_f)}{Nf} \sum_{\substack{n=1 \\ f\mid n}}^N n q^{-n/2} =: \frac{c}{Nf}  \sum_{\substack{n=1 \\ f\mid n}}^N n q^{-n/2}.
    \]
    Since $q^{-1/2}<1$, we obtain
    \[
    | R_N - S_N | \leq \frac{c}{Nf} \sum_{\substack{n=1 \\ f\mid n}}^{+\infty} n q^{-n/2} =\frac{c}{Nf} \sum_{m=0}^{+\infty} fm q^{-fm/2} = \frac{c}{N}  \frac{q^{-f/2}}{(q^{-f/2}-1)^2} =: \frac{c_3}{N}.
    \]
    This completes the proof of the bound. Letting $N$ tend to infinity shows that the set $R_q(a,d)$ has $d_3$-density equal to $\delta_q(a,d)$.
\end{proof}

\begin{corollary}
    The set $R_q(a,d)$ has $d_4$ and Dirichlet density equal to $\delta_q(a,d)$.
\end{corollary}

\begin{proof}
    Theorem \ref{MAINTHM} establishes the existence and the value of the $d_3$-density of $R_q(a,d)$. The result follows from \cite[Theorem A]{Ba3}.
\end{proof}


\section{Closed-form for the \texorpdfstring{$d_3$}{TEXT}-density}\label{SectionClosedD3}

We proved in Section \ref{SectionMainTheorems} that $\delta_q(a,d)$ is the $d_3$-density of $R_q(a,d)$. However, this constant is defined via a series, meaning that there are infinitely many operations to carry out in order to compute it. We show in this section, under the assumption $f_{u,v}=\mathrm{ord}_{dv}(q)$ for all $u\mid d$ and $v\mid d^\infty$, that $\delta_q(a,d)$ can be written in a closed-form formula, that is, a formula that requires only finitely many simple operations. We define a function $\eta : \mathbb{N}^2\longrightarrow \mathbb{N}$ by $\eta(m,n)=2^{v_2(q^{\Bar{f}}+1)-1}$ if $2\nmid (m,n)$ and $[\mathcal{P}]=1$, and $\eta(m,n)=1$ otherwise. Note that $m\mapsto \eta(m,n)$ is a multiplicative function for all $n\geq 1$.

\begin{proposition}\label{PropositionClosedFormDelta_w}
    Assume that $f_{u,v}=\mathrm{ord}_{dv}(q)$ for all $u\mid d$ and $v\mid d^\infty$. For each $w\mid d^\infty$, we have
    \[
    \delta_w =
    \begin{dcases}
        \sum_{u\mid d} \frac{\mu(u)(dh,u^\infty)}{u(q^{\Bar{f}}-1,u^\infty)(\nu,u^\infty)\eta(\nu,u)}, & \text{if } fw=\Bar{f}\nu, \, \nu \mid d^\infty; \\
        0, & \text{otherwise}.
    \end{dcases}
    \]
    If $fw=\Bar{f}\nu$, we may denote $\delta_w$ by $\delta(\nu)$ when it is written in the above form to make the dependence in $\nu$ more obvious.
\end{proposition}

\begin{proof}
    The proof being quite similar to the proof of \cite[Lemma 4]{Moree}, we may skip a few details. Since $v\mid e_{fw}$ if and only if $\mathrm{ord}_{dv}(q)\mid fw$, and because $f_{u,v}=\mathrm{ord}_{dv}(q)$, we find from \eqref{EqFormulaDelta_w} that
    \[
    \delta_w = \sum_{v\mid e_{fw}} \sum_{u\mid d} \frac{\mu(u)(uv,h)}{uv}= \sum_{v\mid e_{fw}} \frac{(v,h)}{v} \prod_{l\mid d} \left( 1-\frac{(lv,h)}{l(v,h)} \right),
    \]
    where we used the Euler product formula on the inner sum. Note that the product is non-zero if and only if $(h,d^\infty) \mid e_{fw}$, or equivalently, $\Bar{f}\mid fw$. Assume $fw=\Bar{f}\nu$ for some $\nu\mid d^\infty$. Then, the product is equal to $\varphi(d)/d$. We use the Euler product formula twice on the remaining sum to obtain
    \[
    \delta_w =\prod_{l\mid d} \bigg( 1 - \frac{l^{v_l(dh)}}{l^{v_l(q^{\Bar{f}\nu}-1)+1}}\bigg) =\sum_{u\mid d} \frac{\mu(u)(dh,u^\infty)}{u(q^{\Bar{f}\nu}-1,u^\infty)}.
    \]
    Finally, we apply Lemma \ref{lemma(dinf)} to $(q^{\Bar{f}\nu}-1,u^\infty)$ in the general term of the sum, which is allowed since $u\mid q^{\Bar{f}}-1$, to obtain $\delta_w = \delta(\nu)$.
\end{proof}

\begin{theorem}\label{THMClosedFormd3}
    Put $C=3\cdot 4^{-1}+ 2^{-v_2(q^{\Bar{f}}+1)-1}$. Assume that $f_{u,v}=\mathrm{ord}_{dv}(q)$ for all $u\mid d$ and $v\mid d^\infty$. Then, we have
    \[
    \delta_q(a,d) = \frac{1}{\Bar{f}} \prod_{l\mid d} \bigg(1-\frac{l^{v_l(dh)}C^{[l=2]\cdot [\mathcal{P}]}}{(l+1)l^{v_l(q^{\Bar{f}}-1)}} \bigg).
    \]
\end{theorem}

\begin{proof}
    By Proposition \ref{PropositionClosedFormDelta_w}, we may only consider the indices $w\mid d^\infty$ that satisfy $fw=\Bar{f}\nu$ for some $\nu\mid d^\infty$ in the expression \eqref{delta_q(a,d) first expression} of $\delta_q(a,d)$. We obtain
    \[
    \delta_q(a,d) = \frac{\varphi(d)}{d} \sum_{w\mid d^\infty} \frac{\delta_w}{fw} = \frac{\varphi(d)}{d} \sum_{\nu \mid d^\infty} \frac{\delta(\nu)}{\Bar{f}\nu}.
    \]
    Since the series is absolutely convergent, we may interchange the sum in $\delta(\nu)$ and the series. We find
    \[
    \delta_q(a,d) = \frac{\varphi(d)}{d} \sum_{u\mid d} \frac{\mu(u)(dh,u^\infty)}{u(q^{\Bar{f}}-1,u^\infty)} \sum_{\nu \mid d^\infty} \frac{1}{\nu(\nu,u^\infty)\eta(\nu,u)}.
    \]
    Let $S(u)$ denote the inner series. The function $\nu \mapsto \nu(\nu,u^\infty)\eta(\nu,u)$ is multiplicative. Moreover, we see that $\eta(\nu,u) = \eta(\mathrm{rad}(\nu),u)$, where $\mathrm{rad}(\nu)$ is the radical of $\nu$. By the Euler product formula, we obtain
    \begin{align*}
        S(u)= \prod_{l \mid d} \bigg( 1 + \sum_{r=1}^{+\infty} \frac{1}{l^{r(1+[l\mid u])}\eta(l,u)} \bigg) =  \prod_{l \mid d} \bigg( 1+ \frac{1}{\eta(l,u)(l^{1+[l\mid u]}-1)} \bigg).
    \end{align*}
    If $u$ is odd or $[\mathcal{P}]=0$, then 
    \[
    S(u) =  \prod_{\substack{ l \mid d \\ l\nmid u}} \bigg( 1+\frac{1}{l-1} \bigg) \prod_{l\mid u } \left(1+\frac{1}{l^2-1}\right) = \frac{du}{\varphi(d)\psi(u)}.
    \]
    When $2\mid u$, since $\eta(l,u) = 2^{v_2(q^{\Bar{f}}+1)-1}$ if and only if $l=2$, we have
    \begin{align*}
        S(u) = \prod_{\substack{ l \mid d \\ l\nmid u}} \bigg(\frac{l}{l-1}\bigg) \prod_{l\mid u } \bigg(\frac{l^2}{l^2-1}\bigg) \cdot \frac{3}{4} \bigg(1+ \frac{1}{3\cdot 2^{v_2(q^{\Bar{f}}-1)-1}}\bigg),
    \end{align*}
    which yields $S(u) =duC/\varphi(d)\psi(u)$. The expression of $\delta_q(a,d)$ becomes
    \[
    \delta_q(a,d) = \frac{1}{\Bar{f}} \sum_{u\mid d} \frac{\mu(u)(dh,u^\infty)C^{[2\mid u]\cdot [\mathcal{P}]}}{(q^{\Bar{f}}-1,u^\infty)\psi(u)}.
    \]
    The general term of the sum defines a multiplicative function in the variable $u$, and the result follows from the Euler product formula.
\end{proof}

Note that Theorem \ref{THMClosedFormd3} coincides with \cite[Theorem 3.3]{Ba1} and \cite[Theorem 11]{Ba2} in the case $a=T$ and, respectively, $d=2$ and $d$ an odd prime.

In conclusion, we proved that the set $R_q(a,d)$ has $d_3$-density $\delta_q(a,d)$ in Section \ref{SectionMainTheorems} and that $\delta_q(a,d)$ can be written in a closed form in Section \ref{SectionClosedD3}. However, the latter was done under the assumption that $[\mathbb{F}_{dv,uv}:\mathbb{F}_q]$ is equal to $\mathrm{ord}_{dv}(q)$ for all $v\mid d^\infty$ and $u\mid d$. It is quite easy to prove that this condition holds when $(d,h)=1$ or $(d,m)=1$, where $m=\mathrm{ord}_{\mathbb{F}_q^\times}(\lambda)$. However, this is not enough to cover all cases. For instance, it does not give any information about the density when $(d,h,m)>1$.

\end{document}